\documentclass[11pt]{article}

\usepackage[a4paper,margin=1in]{geometry}
\usepackage{amsmath,amssymb,amsthm,amsfonts}
\usepackage{mathtools}
\usepackage[colorlinks=true,linkcolor=blue,citecolor=blue,urlcolor=blue]{hyperref}

\newtheorem{theorem}{Theorem}[section]
\newtheorem{lemma}[theorem]{Lemma}

\newtheorem{corollary}[theorem]{Corollary}

\theoremstyle{definition}
\newtheorem{definition}[theorem]{Definition}

\theoremstyle{remark}
\newtheorem{remark}[theorem]{Remark}

\newcommand{\IS}{\operatorname{IS}}
\newcommand{\WIS}{\operatorname{WIS}}
\newcommand{\lub}{\operatorname{lub}}
\newcommand{\succop}{\operatorname{succ}}

\title{
Bourbaki--Zorn Normal Forms for Maximality Arguments
}

\author{You-Chang Liu}

\date{}

\begin{document}

\maketitle

\begin{abstract}
We isolate a normal-form mechanism underlying Bourbaki--Witt fixed-point arguments and
least-upper-bound versions of Zorn-type maximality principles. Given a progressive
self-map on a partially ordered set, we define a Bourbaki tower as a well-ordered
trajectory whose successor stages are generated by the map and whose limit stages are
given by least upper bounds of earlier stages. We prove that least upper bounds for
nonempty well-ordered subsets are sufficient to force a fixed point for every progressive
self-map. Thus the fixed-point statement is obtained under a weaker completeness
hypothesis than the usual chain-complete form of the Bourbaki--Witt theorem.

The proof proceeds by constructing a largest Bourbaki tower. The least upper bound of
this largest tower belongs to the tower itself and is a fixed point of the map. As a
consequence, strictly progressive self-maps cannot exist in such posets. Combining this
obstruction with a choice selector on strict upper cones yields a concise maximality
principle: if every nonempty well-ordered subset has a least upper bound, then the poset
has a maximal element.

The contribution is methodological rather than axiomatic. The paper makes explicit a
reusable proof architecture connecting well-ordered Bourbaki--Witt fixed points, strict
progression obstructions, and least-upper-bound versions of Zorn-type maximality
arguments.
\end{abstract}

\tableofcontents

\section{Introduction}

The Bourbaki--Witt fixed-point theorem is often formulated for partially ordered sets in
which every chain has a least upper bound. Under that hypothesis, every progressive
self-map has a fixed point. Zorn's Lemma, by contrast, is usually presented as a
maximality principle: if every chain in a partially ordered set has an upper bound, then
the poset has a maximal element. Although these principles are closely related, their
proofs are often written in different languages.

The purpose of this paper is to isolate a common normal form behind these arguments. The
main fixed-point result uses only least upper bounds of nonempty well-ordered subsets,
rather than least upper bounds of all chains. This is enough because the trajectories
generated by the tower construction are well-ordered. Thus the fixed-point theorem proved
below may be viewed as a well-ordered-subset refinement of the usual chain-complete
Bourbaki--Witt theorem.

The central object is a Bourbaki tower: a well-ordered trajectory generated by a
progressive map, whose successor stages are determined by the map and whose limit stages
are determined by least upper bounds of earlier stages. The tower construction makes
explicit the transfinite mechanism by which a progressive process is forced to terminate:
one iterates the map at successor stages and takes least upper bounds at limit stages.

The main observation is that maximality can be reformulated as the impossibility of a
strictly progressive self-map. If a poset has no maximal element, a choice function selects,
above each element, a strictly larger successor. This gives a strictly progressive self-map.
But the fixed-point theorem proved below forces every progressive self-map to have a fixed point,
contradicting strict progressiveness. Hence a maximal element must exist.

The novelty of the formulation lies in isolating the Bourbaki tower as a canonical normal
form for fixed-point and maximality arguments. In this perspective, fixed points arise
from the forced termination of well-ordered progressive trajectories, while maximality
arises from the impossibility of strict progression under the same least-upper-bound
hypothesis.

\subsection*{Contribution}

The contribution of this paper is to isolate a well-ordered tower normal form for
Bourbaki--Witt type fixed-point arguments and least-upper-bound maximality principles.
The fixed-point theorem is proved using least upper bounds only for nonempty
well-ordered subsets, rather than for all chains. This reflects the fact that the
trajectories generated by the construction are themselves well-ordered.

The same tower mechanism also gives a Zorn-type maximality principle under a
least-upper-bound hypothesis. Absence of maximal elements produces a strictly progressive
self-map, while the tower theorem forces every progressive self-map to have a fixed point.
Thus maximality is obtained as a fixed-point obstruction to strict progression.

\section{Preliminaries}

We begin with the order-theoretic terminology used throughout the paper.

\begin{definition}[Partially ordered set]
A \emph{partially ordered set}, or \emph{poset}, is a set \(X\) equipped with a relation
\(\leq\) satisfying:
\begin{enumerate}
\item \(x\leq x\) for all \(x\in X\);
\item if \(x\leq y\) and \(y\leq x\), then \(x=y\);
\item if \(x\leq y\) and \(y\leq z\), then \(x\leq z\).
\end{enumerate}
We write \(x<y\) to mean \(x\leq y\) and \(x\neq y\).
\end{definition}

\begin{definition}[Well-ordered subset]
Let \((X,\leq)\) be a poset. A subset \(Y\subseteq X\) is called \emph{well-ordered}
if the order induced from \(X\) is total on \(Y\) and every nonempty subset of \(Y\) has a
least element.
\end{definition}

\begin{definition}[Initial segment determined by an element]
Let \(Y\) be a well-ordered set and let \(y\in Y\). The \emph{initial segment} of \(Y\)
determined by \(y\) is
\[
\IS_Y(y)=\{z\in Y\mid z<y\}.
\]
The \emph{weak initial segment} of \(Y\) determined by \(y\) is
\[
\WIS_Y(y)=\{z\in Y\mid z\leq y\}.
\]
\end{definition}

\begin{definition}[Initial segment of a well-ordered subset]
Let \(A\) and \(B\) be well-ordered subsets of a poset \((X,\leq)\). We say that
\(A\) is an \emph{initial segment} of \(B\) if \(A\subseteq B\) and, whenever
\(a\in A\) and \(b\in B\) satisfy
\[
b<a,
\]
one has \(b\in A\).

If \(A\neq B\), then \(A\) is called a \emph{proper initial segment} of \(B\).
\end{definition}

\begin{definition}[Successor]
Let \(Y\) be a well-ordered set and let \(y\in Y\). If the set
\[
\{z\in Y\mid y<z\}
\]
is nonempty, its least element is called the \emph{successor} of \(y\) in \(Y\), denoted
\[
\succop_Y(y).
\]
If this set is empty, then \(y\) has no successor in \(Y\); equivalently, \(y\) is the largest
element of \(Y\).
\end{definition}

\begin{definition}[Successor and limit elements]
Let \(Y\) be a well-ordered set and let \(y\in Y\). We say that \(y\) is a
\emph{successor element} of \(Y\) if there exists \(z\in Y\) such that
\[
y=\succop_Y(z).
\]
An element \(y\in Y\) is called a \emph{limit element} of \(Y\) if \(y\neq \min Y\) and
\(y\) is not a successor element of \(Y\).
\end{definition}

\begin{definition}[Least upper bound]
Let \((X,\leq)\) be a poset and let \(A\subseteq X\). An element \(u\in X\) is called a
\emph{least upper bound} of \(A\), written
\[
u=\lub_X(A),
\]
if:
\begin{enumerate}
\item \(a\leq u\) for all \(a\in A\);
\item whenever \(v\in X\) satisfies \(a\leq v\) for all \(a\in A\), one has \(u\leq v\).
\end{enumerate}
\end{definition}

\begin{remark}[Transfinite terminology]
In this paper, a \emph{transfinite} construction means a construction carried out along a
well-ordered set of stages, not merely along the natural numbers. Thus a tower may begin
with ordinary successor steps
\[
x_0,\quad f(x_0),\quad f^2(x_0),\quad \ldots,
\]
but it may also have limit stages. At a limit stage, there is no immediate predecessor;
instead, the new element is defined as the least upper bound of all earlier elements. In a
Bourbaki tower this means that, for a limit element \(y\),
\[
y=\lub_X(\IS_Y(y)).
\]
Thus ``transfinite'' refers to the use of both successor stages and limit stages in a
well-ordered construction.
\end{remark}

\section{Progressive Maps and Bourbaki Towers}

The central object in this paper is a well-ordered trajectory generated by a progressive
self-map.

\begin{definition}[Progressive map]
Let \((X,\leq)\) be a partially ordered set. A map
\[
f:X\to X
\]
is called \emph{progressive} if
\[
x\leq f(x)
\qquad\text{for all }x\in X.
\]
It is called \emph{strictly progressive} if
\[
x<f(x)
\qquad\text{for all }x\in X.
\]
\end{definition}

\begin{definition}[Bourbaki tower]
Let \((X,\leq)\) be a partially ordered set, let \(f:X\to X\) be progressive, and fix
\(x_0\in X\).

A subset \(Y\subseteq X\) is called a \emph{Bourbaki \(f\)-tower based at \(x_0\)} if the
following conditions hold:
\begin{enumerate}
\item \(Y\) is well-ordered by the order induced from \(X\), and \(x_0\) is the least
element of \(Y\);

\item if \(y\in Y\) is not the largest element of \(Y\), then its successor in \(Y\) is
given by
\[
\succop_Y(y)=f(y);
\]

\item if \(y\in Y\) is a limit element of \(Y\), then
\[
y=\lub_X(\IS_Y(y)).
\]
\end{enumerate}
\end{definition}

\begin{remark}
Condition (2) says that successor stages are generated by the map \(f\). Condition (3)
says that genuine limit stages are determined canonically by least upper bounds of all
earlier stages. Thus a Bourbaki tower is a well-ordered trajectory of \(f\), enlarged at
limit stages by the least upper bounds forced by the order structure of \(X\).
\end{remark}

\section{The Tower Comparison Lemma}

The key technical fact is that two Bourbaki towers based at the same initial point cannot
branch. They must agree on their common part, and one tower must be an initial segment
of the other.

\begin{lemma}[Tower comparison lemma]
Let \((X,\leq)\) be a poset, let \(f:X\to X\) be progressive, and let \(Y,Y'\) be Bourbaki
\(f\)-towers based at the same element \(x_0\). Then one of \(Y\) and \(Y'\) is an initial
segment of the other.
\end{lemma}

\begin{proof}
Define
\[
V=\{y\in Y\cap Y'\mid \WIS_Y(y)=\WIS_{Y'}(y)\}.
\]
Thus \(V\) is the set of common elements up to which the two towers have exactly the
same weak initial segment. Since both towers are based at \(x_0\), we have \(x_0\in V\).
Hence \(V\neq\varnothing\).

We first observe that \(V\) is an initial segment of both \(Y\) and \(Y'\). Indeed, suppose
\(y\in V\) and \(u\in Y\) satisfy \(u<y\). Since
\[
\WIS_Y(y)=\WIS_{Y'}(y),
\]
we have \(u\in Y'\). Moreover, the weak initial segments determined by \(u\) also agree.
Indeed, if \(t\in\WIS_Y(u)\), then \(t\leq u<y\), so
\[
t\in\WIS_Y(y)=\WIS_{Y'}(y),
\]
and hence \(t\in\WIS_{Y'}(u)\). The reverse inclusion is proved in the same way. Thus
\[
\WIS_Y(u)=\WIS_{Y'}(u),
\]
so \(u\in V\). The same argument with \(Y\) and \(Y'\) interchanged shows that \(V\) is
also an initial segment of \(Y'\).

We now distinguish two cases.

\medskip

\noindent
\emph{Case 1: \(V\) has a largest element \(v\).}

If \(v\) is the largest element of \(Y\), then
\[
Y=\WIS_Y(v)=\WIS_{Y'}(v),
\]
so \(Y\) is an initial segment of \(Y'\). Similarly, if \(v\) is the largest element of
\(Y'\), then \(Y'\) is an initial segment of \(Y\).

Assume, therefore, that \(v\) is not the largest element of either tower. By the successor
condition for Bourbaki towers,
\[
\succop_Y(v)=f(v)=\succop_{Y'}(v).
\]
Since a successor is strictly larger than the element whose successor it is, we have
\[
v<f(v).
\]
Furthermore,
\[
\WIS_Y(f(v))
=
\WIS_Y(v)\cup\{f(v)\}
=
\WIS_{Y'}(v)\cup\{f(v)\}
=
\WIS_{Y'}(f(v)).
\]
Thus \(f(v)\in V\), contradicting the maximality of \(v\) in \(V\). Hence, in this case,
one of \(Y\) and \(Y'\) is an initial segment of the other.

\medskip

\noindent
\emph{Case 2: \(V\) has no largest element.}

We claim that at least one of \(Y\) and \(Y'\) is equal to \(V\). Suppose, for
contradiction, that neither is equal to \(V\). Let \(y\) be the least element of
\(Y\setminus V\), and let \(y'\) be the least element of \(Y'\setminus V\). Since \(V\) is
an initial segment of both towers and has no largest element, we have
\[
V=\IS_Y(y)=\IS_{Y'}(y').
\]

The element \(y\) cannot be a successor element of \(Y\). For if
\[
y=\succop_Y(u)
\]
for some \(u\in Y\), then \(u<y\), hence \(u\in V\). Since \(V\) has no largest element,
there exists \(w\in V\) such that \(u<w\). But \(w\in V=\IS_Y(y)\), so
\[
u<w<y,
\]
contradicting the fact that \(y\) is the successor of \(u\) in \(Y\). Hence \(y\) is a
limit element of \(Y\). Similarly, \(y'\) is a limit element of \(Y'\).

By the limit condition for Bourbaki towers,
\[
y=\lub_X(\IS_Y(y))
\qquad\text{and}\qquad
y'=\lub_X(\IS_{Y'}(y')).
\]
Since
\[
\IS_Y(y)=V=\IS_{Y'}(y'),
\]
we obtain
\[
y=\lub_X(V)=y'.
\]
Therefore
\[
\WIS_Y(y)=V\cup\{y\}=V\cup\{y'\}=\WIS_{Y'}(y'),
\]
so \(y\in V\), contradicting \(y\in Y\setminus V\).

Thus at least one of \(Y\) and \(Y'\) is equal to \(V\). Since \(V\) is an initial segment
of both towers, one of \(Y\) and \(Y'\) is an initial segment of the other.
\end{proof}

\section{The Well-Ordered Bourbaki--Witt Normal Form}

We now prove the main structural theorem. It is a well-ordered-subset version of the
Bourbaki--Witt fixed-point argument: instead of assuming least upper bounds for all
chains, it only uses least upper bounds for the well-ordered trajectories generated by the
construction.

\begin{theorem}[Well-ordered Bourbaki--Witt normal form]
Let \((X,\leq)\) be a nonempty partially ordered set such that every nonempty
well-ordered subset of \(X\) admits a least upper bound in \(X\). Let \(f:X\to X\) be
progressive, and fix \(x_0\in X\).

Then there exists a largest Bourbaki \(f\)-tower based at \(x_0\), denoted
\[
\Omega_f(x_0).
\]
Moreover, if
\[
\omega=\lub_X(\Omega_f(x_0)),
\]
then
\[
\omega\in \Omega_f(x_0)
\]
and
\[
f(\omega)=\omega.
\]
In particular, every progressive self-map on \(X\) has a fixed point.
\end{theorem}

\begin{proof}
Let \(\mathcal T\) be the collection of all Bourbaki \(f\)-towers based at \(x_0\). The
singleton set \(\{x_0\}\) belongs to \(\mathcal T\), so \(\mathcal T\neq\varnothing\).

By the Tower Comparison Lemma, any two elements of \(\mathcal T\) are comparable by
the initial-segment relation. Define
\[
\Omega_f(x_0)=\bigcup_{Y\in\mathcal T}Y.
\]

We verify the required properties in stages.

\begin{enumerate}
\item \emph{\(\Omega_f(x_0)\) is well-ordered.}

Since any two members of \(\mathcal T\) are comparable by the initial-segment relation,
the orders on the towers are compatible. Hence the order induced from \(X\) is total on
\(\Omega_f(x_0)\).

Let \(A\subseteq \Omega_f(x_0)\) be nonempty. Choose \(a\in A\), and choose
\(Y_a\in\mathcal T\) such that \(a\in Y_a\). Since \(A\cap Y_a\) is a nonempty subset of
the well-ordered set \(Y_a\), it has a least element; call it \(a_0\).

We claim that \(a_0\) is the least element of \(A\). Let \(b\in A\). If \(b\in Y_a\), then
\(a_0\leq b\) by the definition of \(a_0\). If \(b\notin Y_a\), choose
\(Y_b\in\mathcal T\) such that \(b\in Y_b\). By the Tower Comparison Lemma, \(Y_a\) and
\(Y_b\) are comparable by the initial-segment relation. Since \(b\notin Y_a\), the only
possible case is that \(Y_a\) is a proper initial segment of \(Y_b\). Hence every element
of \(Y_a\) is strictly below \(b\). In particular,
\[
a_0<b.
\]
Thus \(a_0\leq b\) for every \(b\in A\), so \(a_0\) is the least element of \(A\).
Therefore \(\Omega_f(x_0)\) is well-ordered. Its least element is \(x_0\).

\item \emph{\(\Omega_f(x_0)\) satisfies the successor condition.}

Let \(y\in\Omega_f(x_0)\), and suppose that \(y\) is not the largest element of
\(\Omega_f(x_0)\). Then there exists \(z\in\Omega_f(x_0)\) such that \(y<z\). Choose
towers \(Y,Z\in\mathcal T\) such that \(y\in Y\) and \(z\in Z\). By the Tower Comparison
Lemma, \(Y\) and \(Z\) are comparable by the initial-segment relation. Since \(y<z\), one
of these two towers contains both \(y\) and \(z\). In that tower, \(y\) is not largest, and
its successor is \(f(y)\). Hence
\[
f(y)\in\Omega_f(x_0).
\]

Now let \(u\in\Omega_f(x_0)\) satisfy \(y<u\). By the same comparability argument, there
is a tower in \(\mathcal T\) containing both \(y\) and \(u\). In that tower, the least
element strictly above \(y\) is \(f(y)\). Therefore
\[
f(y)\leq u.
\]
Thus \(f(y)\) is the least element of \(\Omega_f(x_0)\) strictly above \(y\). Hence
\[
\succop_{\Omega_f(x_0)}(y)=f(y).
\]

\item \emph{\(\Omega_f(x_0)\) satisfies the limit condition.}

Let \(y\) be a limit element of \(\Omega_f(x_0)\). Choose \(Y\in\mathcal T\) such that
\(y\in Y\). We claim that \(y\) is a limit element of \(Y\).

First, \(y\) cannot be the least element of \(Y\), because then \(y=x_0\), which is also
the least element of \(\Omega_f(x_0)\), contrary to the assumption that \(y\) is a limit
element of \(\Omega_f(x_0)\). Next, suppose that \(y\) is a successor element of \(Y\),
say
\[
y=\succop_Y(u).
\]
We show that \(y\) would then be a successor element of \(\Omega_f(x_0)\), giving a
contradiction. Let \(v\in\Omega_f(x_0)\) satisfy
\[
u<v<y.
\]
Choose \(Z\in\mathcal T\) such that \(v\in Z\). By the Tower Comparison Lemma, \(Y\)
and \(Z\) are comparable by the initial-segment relation. If \(Z\) is an initial segment of
\(Y\), then \(v\in Y\), contradicting the fact that \(y\) is the successor of \(u\) in
\(Y\). If \(Y\) is an initial segment of \(Z\), then \(u,y\in Z\), and since \(v<y\), the
initial-segment property forces \(v\in Y\), again a contradiction. Thus there is no element
of \(\Omega_f(x_0)\) strictly between \(u\) and \(y\). Hence \(y\) is the successor of
\(u\) in \(\Omega_f(x_0)\), contradicting the assumption that \(y\) is a limit element of
\(\Omega_f(x_0)\). Therefore \(y\) is a limit element of \(Y\).

We now compare initial segments. If \(v\in\IS_{\Omega_f(x_0)}(y)\), choose
\(Z\in\mathcal T\) such that \(v\in Z\). By the Tower Comparison Lemma, \(Y\) and \(Z\)
are comparable by the initial-segment relation. If \(Z\) is an initial segment of \(Y\),
then \(v\in Y\). If \(Y\) is an initial segment of \(Z\), then \(y\in Y\subseteq Z\), and
\(v<y\) implies \(v\in Y\). Hence \(v\in\IS_Y(y)\). The reverse inclusion is immediate.
Therefore
\[
\IS_{\Omega_f(x_0)}(y)=\IS_Y(y).
\]
Since \(Y\) is a Bourbaki tower and \(y\) is a limit element of \(Y\),
\[
y=\lub_X(\IS_Y(y)).
\]
Consequently,
\[
y=\lub_X(\IS_{\Omega_f(x_0)}(y)).
\]

\item \emph{\(\Omega_f(x_0)\) is the largest Bourbaki tower.}

By the preceding three steps, \(\Omega_f(x_0)\) is a Bourbaki \(f\)-tower. Therefore
\[
\Omega_f(x_0)\in\mathcal T.
\]
It is clearly the largest element of \(\mathcal T\) under inclusion.

\item \emph{The least upper bound of \(\Omega_f(x_0)\) belongs to the tower.}

Since \(\Omega_f(x_0)\) is a nonempty well-ordered subset of \(X\), the hypothesis gives
a least upper bound in \(X\). Let
\[
\omega=\lub_X(\Omega_f(x_0)).
\]

Suppose first that \(\omega\notin\Omega_f(x_0)\). Since \(\omega\) is the least upper
bound of \(\Omega_f(x_0)\), every element of \(\Omega_f(x_0)\) is strictly below
\(\omega\). Moreover, \(\Omega_f(x_0)\) cannot have a largest element; otherwise its
largest element would already be the least upper bound of \(\Omega_f(x_0)\), contradicting
\(\omega\notin\Omega_f(x_0)\).

Adjoin \(\omega\) to \(\Omega_f(x_0)\) as a new largest element. The resulting set is
well-ordered. Since \(\Omega_f(x_0)\) has no largest element, the newly adjoined element
\(\omega\) is a limit element of the enlarged tower, and
\[
\omega=\lub_X(\Omega_f(x_0)).
\]
All earlier successor and limit conditions are unchanged. Hence the enlarged set is again
a Bourbaki \(f\)-tower based at \(x_0\), contradicting the maximality of
\(\Omega_f(x_0)\). Therefore
\[
\omega\in\Omega_f(x_0).
\]

\item \emph{The element \(\omega\) is a fixed point of \(f\).}

Suppose that
\[
f(\omega)>\omega.
\]
Since \(\omega\) is an upper bound of \(\Omega_f(x_0)\), every element of
\(\Omega_f(x_0)\) is below \(\omega\), and hence strictly below \(f(\omega)\). Therefore
\[
\Omega_f(x_0)\cup\{f(\omega)\}
\]
is obtained by adjoining \(f(\omega)\) as a new largest element.

In the enlarged set, \(f(\omega)\) is the successor of \(\omega\). The successor condition
therefore holds at \(\omega\), since its new successor is precisely \(f(\omega)\). The
newly added element \(f(\omega)\) is a successor element, so no new limit condition has to
be verified at that stage. All earlier successor and limit conditions are unchanged. Hence
the enlarged set is again a Bourbaki \(f\)-tower based at \(x_0\), contradicting the
maximality of \(\Omega_f(x_0)\).

Since \(f\) is progressive, we already have
\[
\omega\leq f(\omega).
\]
The strict inequality \(f(\omega)>\omega\) is impossible, and therefore
\[
f(\omega)=\omega.
\]
Thus \(f\) has a fixed point.
\end{enumerate}
\end{proof}

\begin{remark}
The usual chain-complete form of the Bourbaki--Witt theorem assumes that every chain in
\(X\) has a least upper bound. The theorem above requires only least upper bounds for
nonempty well-ordered subsets. Since every well-ordered subset is a chain, the
chain-complete hypothesis implies the hypothesis used here, but not conversely in general.
Thus the fixed-point conclusion is obtained under a weaker completeness assumption.
\end{remark}

\begin{corollary}[No strict progression principle]
Let \((X,\leq)\) be a nonempty partially ordered set such that every nonempty
well-ordered subset of \(X\) admits a least upper bound. Then there is no strictly
progressive map
\[
f:X\to X.
\]
\end{corollary}

\begin{proof}
Suppose, for contradiction, that \(f:X\to X\) is strictly progressive. Since \(X\) is
nonempty, choose \(x_0\in X\). By the preceding theorem applied to
\(f\) and \(x_0\), the map \(f\) has a fixed point \(\omega\in X\). Hence
\[
f(\omega)=\omega.
\]
This contradicts strict progressiveness, which gives
\[
\omega<f(\omega).
\]
Therefore no strictly progressive self-map \(X\to X\) exists.
\end{proof}

\section{The Bourbaki--Zorn Machine}

We now explain how a choice selector on strict upper cones can be converted into a
maximality argument through the tower normal form. The key idea is that, if a poset has
no maximal element, then choice selects a strictly larger element above each point. This
produces a strictly progressive self-map. The preceding fixed-point theorem then
forces such a map to have a fixed point, which is impossible.

\begin{definition}[Strict upper cone]
Let \((X,\leq)\) be a partially ordered set. For \(x\in X\), define the strict upper cone of
\(x\) by
\[
U_x=\{y\in X\mid x<y\}.
\]
\end{definition}

\begin{theorem}[Bourbaki--Zorn fixed-point obstruction]
Assume that every nonempty indexed family of nonempty sets admits a choice function.

Let \((X,\leq)\) be a nonempty partially ordered set such that every nonempty
well-ordered subset of \(X\) admits a least upper bound in \(X\). Then \(X\) has a
maximal element.
\end{theorem}

\begin{proof}
Suppose, for contradiction, that \(X\) has no maximal element. Then for every \(x\in X\),
the strict upper cone
\[
U_x=\{y\in X\mid x<y\}
\]
is nonempty.

By the choice principle applied to the indexed family \((U_x)_{x\in X}\), choose a map
\[
s:X\to X
\]
such that
\[
s(x)\in U_x
\qquad\text{for every }x\in X.
\]
Define \(f:X\to X\) by
\[
f(x)=s(x).
\]
Then
\[
x<f(x)
\qquad\text{for every }x\in X,
\]
so \(f\) is strictly progressive.

This contradicts the No Strict Progression Principle, which says that no strictly
progressive self-map can exist on \(X\) under the stated least-upper-bound hypothesis.
Therefore \(X\) must contain a maximal element.
\end{proof}

\begin{remark}[Bourbaki--Zorn normal form]
The preceding proof factors through the following mechanism:
\[
\begin{aligned}
\text{choice selector}
&\Longrightarrow
\text{strictly progressive self-map} \\
&\Longrightarrow
\text{Bourbaki tower} \\
&\Longrightarrow
\text{fixed-point contradiction}.
\end{aligned}
\]
Indeed, if \(X\) has no maximal element, choice produces a map \(f:X\to X\) satisfying
\(x<f(x)\) for every \(x\in X\). The preceding fixed-point theorem then produces a
largest Bourbaki tower \(\Omega_f(x_0)\), whose least upper bound \(\omega\) belongs to
the tower and satisfies
\[
f(\omega)=\omega.
\]
This contradicts strict progressiveness. Thus maximality is obtained as a fixed-point
obstruction to strict progression.
\end{remark}

\section{Relation with Zorn's Lemma}

The fixed-point theorem proved above is weaker in its completeness hypothesis than the
usual chain-complete form of the Bourbaki--Witt theorem: it requires least upper bounds
only for nonempty well-ordered subsets. The resulting maximality principle, however,
should be compared with Zorn's Lemma more carefully.

The usual form of Zorn's Lemma assumes that every chain \(C\subseteq X\) has an upper
bound:
\[
\forall C\subseteq X,\quad
C\text{ a chain}
\Longrightarrow
\exists u\in X\text{ such that }c\leq u\text{ for all }c\in C.
\]
By contrast, the maximality principle obtained here assumes that every nonempty
well-ordered subset has a least upper bound:
\[
\forall W\subseteq X,\quad
W\text{ nonempty and well-ordered}
\Longrightarrow
\exists \lub_X(W)\in X.
\]
Thus the present maximality principle is not a direct replacement for the usual
chain-upper-bound form of Zorn's Lemma. It is a least-upper-bound version adapted to the
Bourbaki tower construction.

The reason is that the comparison argument requires canonical behavior at limit stages.
If \(Y\) and \(Y'\) are two Bourbaki towers and their preceding initial segments agree,
\[
\IS_Y(y)=\IS_{Y'}(y'),
\]
then the limit-stage condition forces
\[
y=\lub_X(\IS_Y(y))
 =\lub_X(\IS_{Y'}(y'))
 =y'.
\]
This is the canonical step that prevents branching of towers.

If one assumes only that upper bounds exist, then for the same initial segment one may
only know that there is some \(u\in X\) such that
\[
z\leq u
\qquad\text{for all }z\in \IS_Y(y).
\]
Such an upper bound need not be unique, and it need not be least. Two tower constructions
could therefore choose different upper bounds at the same limit stage. To recover the
classical chain-upper-bound form of Zorn's Lemma by the same tower method, one would
need additional choice data selecting upper bounds coherently at limit stages.

The underlying mechanism is nevertheless parallel:
\[
\begin{aligned}
\text{choice}
&\Longrightarrow
\text{successor rule }x\mapsto f(x)>x,\\
&\Longrightarrow
\text{transfinite well-ordered tower},\\
&\Longrightarrow
\text{fixed point }f(\omega)=\omega,\\
&\Longrightarrow
\text{contradiction to strict progressiveness}.
\end{aligned}
\]
In this sense, the paper does not replace the classical Zorn Lemma. Rather, it isolates
the Bourbaki--Witt fixed-point mechanism underlying a least-upper-bound version of
Zorn-type maximality arguments.

\section{Conclusion}

We have isolated a normal-form mechanism underlying Bourbaki--Witt fixed-point arguments
and least-upper-bound versions of Zorn-type maximality principles. The central object is
the Bourbaki tower: a well-ordered trajectory generated by a progressive self-map. Its
successor stages are determined by the map, and its limit stages are determined by least
upper bounds of earlier stages.

The main fixed-point theorem shows that least upper bounds for nonempty well-ordered
subsets are sufficient to force a fixed point for every progressive self-map. This is a
weaker completeness hypothesis than the usual chain-complete formulation of the
Bourbaki--Witt theorem. The proof constructs a largest Bourbaki tower and shows that the
least upper bound of this tower belongs to the tower itself and is fixed by the map.

This yields the No Strict Progression Principle: no strictly progressive self-map can exist
under the stated least-upper-bound hypothesis. Combining this obstruction with a choice
selector on strict upper cones gives a concise Zorn-type maximality principle. The
resulting architecture is:
\[
\begin{aligned}
\text{choice}
&\Longrightarrow
\text{strict progression} \\
&\Longrightarrow
\text{Bourbaki tower} \\
&\Longrightarrow
\text{fixed point} \\
&\Longrightarrow
\text{contradiction}.
\end{aligned}
\]

The contribution of this viewpoint is methodological. It clarifies how Bourbaki--Witt
fixed points, strict progression obstructions, and least-upper-bound maximality principles
are connected through a common transfinite architecture.


\begin{thebibliography}{9}

\bibitem{Bourbaki}
Nicolas Bourbaki,
\emph{Theory of Sets},
Elements of Mathematics,
Springer, 2004.

\bibitem{DaveyPriestley}
B. A. Davey and H. A. Priestley,
\emph{Introduction to Lattices and Order},
2nd ed.,
Cambridge University Press, 2002.

\bibitem{Halmos}
Paul R. Halmos,
\emph{Naive Set Theory},
Undergraduate Texts in Mathematics,
Springer, 1974.

\bibitem{Jech}
Thomas Jech,
\emph{Set Theory},
3rd millennium ed.,
Springer Monographs in Mathematics,
Springer, 2003.

\bibitem{Kelley}
John L. Kelley,
\emph{General Topology},
Graduate Texts in Mathematics, Vol. 27,
Springer, 1975.

\bibitem{Kunen}
Kenneth Kunen,
\emph{Set Theory},
Studies in Logic, Vol. 34,
College Publications, 2011.

\bibitem{MacLane}
Saunders Mac Lane,
\emph{Categories for the Working Mathematician},
2nd ed.,
Graduate Texts in Mathematics, Vol. 5,
Springer, 1998.

\end{thebibliography}
\end{document}